\documentclass[12pt]{amsart}
\usepackage{amssymb,amsmath,amsfonts,latexsym}
\usepackage{bm}
\usepackage[all,cmtip]{xy}
\usepackage{amscd}

\newcommand{\newsection}[1]{\setcounter{equation}{0} \section{#1}}
\newcommand{\bea}{\begin{eqnarray}}
	\newcommand{\eea}{\end{eqnarray}}

\newcommand{\clb}{\mathcal{B}}

\newcommand{\cle}{\mathcal{E}}

\newcommand{\clh}{\mathcal{H}}

\newcommand{\clm}{\mathcal{M}}
\newcommand{\cln}{\mathcal{N}}

\newcommand{\clr}{\mathcal{R}}

\newcommand{\clu}{\mathcal{U}}
\newcommand{\clw}{\mathcal{W}}

\def\textmatrix#1&#2\\#3&#4\\{\bigl({#1 \atop #3}\ {#2 \atop #4}\bigr)}
\def\dispmatrix#1&#2\\#3&#4\\{\left({#1 \atop #3}\ {#2 \atop #4}\right)}
\newcommand{\be}{\begin{equation}}
	\newcommand{\ee}{\end{equation}}
\newcommand{\ben}{\begin{eqnarray*}}
	\newcommand{\een}{\end{eqnarray*}}

\newcommand{\NI}{\noindent}

\newcommand{\bi}{\begin{itemize}}
	\newcommand{\ei}{\end{itemize}}
\newcommand{\diag}{\mbox{diag}}

\theoremstyle{definition}

\theoremstyle{plain}

\newtheorem{thm}{Theorem}[section]

\newtheorem{lem}[thm]{Lemma}

\theoremstyle{definition}
\newtheorem{defn}[thm]{Definition}
\newtheorem{rem}[thm]{Remark}
\newtheorem{ex}[thm]{Example}

\numberwithin{equation}{section}
\let\phi=\varphi

\begin{document}

\title[Wold-type decomposition for $\clu_n$-twisted contractions]
{Wold-type decomposition for $\clu_n$-twisted contractions}	

\author[Majee]{Satyabrata Majee,}
\address{Indian Institute of Technology Roorkee, Department of Mathematics,
	Roorkee-247 667, Uttarakhand, India}
\email{smajee@ma.iitr.ac.in}

\author[Maji]{Amit Maji$^\dag$}
\address{Indian Institute of Technology Roorkee, Department of Mathematics,
	Roorkee-247 667, Uttarakhand,  India}
\email{amit.maji@ma.iitr.ac.in, amit.iitm07@gmail.com ($^\dag$Corresponding author)}


%\today


\begin{abstract}
Let $n>1$, and $\{U_{ij}\}$ for $1 \leq i < j \leq n$ be $\binom{n}{2}$ 
commuting unitaries on a Hilbert space $\clh$ such that $U_{ji}:=U^*_{ij}$.
An $n$-tuple of contractions  $(T_1, \dots, T_n)$ on $\clh$ is called  
$\clu_n$-twisted contraction with respect to a twist $\{U_{ij}\}_{i<j}$ 
if $T_1, \dots, T_n$ satisfy 
\[
T_iT_j=U_{ij}T_jT_i; \hspace{0.5cm}
\hspace{1cm} T_i^*T_j= U^*_{ij}T_jT_i^* \hspace{0.5cm}
\mbox{and} \hspace{0.5cm} T_kU_{ij} =U_{ij}T_k
\]
for all $i,j,k=1, \dots, n$ and $i \neq j$. 
	
\NI
We obtain a recipe to calculate the orthogonal spaces of the Wold-type 
decomposition for $\clu_n$-twisted contractions on Hilbert spaces. As 
a by-product, a new proof as well as complete structure for $\clu_2$-twisted 
(or pair of doubly twisted) and $\clu_n$-twisted isometries have been 
established.
\end{abstract}
\maketitle

\newsection{Introduction }
A fundamental problem in the theory of operators, function theory 
and operator algebras is the classification problem for a tuple 
of commuting isometries on Hilbert spaces. The canonical decomposition for
a contraction plays a significant role in many areas of operator algebras 
and operator theory, namely, dilation theory, invariant subspace theory, 
operator interpolation problem, etc. It says that every contraction 
can be uniquely decomposed into the orthogonal sum of a unitary operator 
and a completely non-unitary operator. In particular, the canonical 
decomposition of an isometry coincides with the classical {\it{Wold 
decomposition or Wold-von Neumann decomposition}}. Indeed, the completely 
non-unitary part of an isometry becomes a unilateral shift (of any multiplicity). 
This decomposition was firstly studied by Wold \cite{WOLD-TIME SERIES} 
for stationary stochastic processes. It is expected that the multidimensional 
Wold-type decomposition will provide a large class of applications.

A natural issue is the extension of decomposition from a single 
contraction to a tuple of contractions. Using Suciu's \cite{SUCIU-SEMIGROUPS} 
decomposition of the semigroup of isometries, S\l oci\'{n}ski \cite{SLOCINSKI-WOLD} 
firstly obtained a Wold-type decomposition for pairs of doubly commuting isometries. 
It states that a pair of doubly commuting isometries have fourfold Wold
-type decomposition of the form unitary-unitary, unitary-shift, 
shift-unitary, and shift-shift. In 2004, Popovici \cite{POPOVICI-WOLD TYPE} 
achieved Wold-type decomposition for a pair $(V_1, V_2)$ 
of commuting isometries on a Hilbert space. More specifically, 
the pair $(V_1, V_2)$ can be uniquely decomposed into the orthogonal sum
of bi-unitary, a shift-unitary, a unitary-shift, and a weak bi-shift.
Later, Sarkar \cite{SARKAR-WOLD} generalized S\l oci\'{n}ski's result 
and also obtained an explicit description of closed subspaces 
in the orthogonal decomposition for the $n$-tuples of 
doubly commuting isometries. Recently Maji, Sarkar, and Sankar 
\cite{MSS-PAIRS} have studied various natural representations 
of a large class of pairs of commuting isometries on Hilbert spaces,
and B\^{i}nzar et al. \cite{BBLPS-Wold-Slocinski} studied 
Wold-S\l oci\'{n}ski decomposition for commuting isometric triples.
On the other hand, power partial isometry is a large class of operators 
with a well-defined completely non-unitary part. Halmos and Wallen 
\cite{HW-POWERS} studied decomposition for a power partial isometry. 
After that Catepill\'{a}n and Szyma\'{n}ski \cite{CS-model}
have generalized for a pair of doubly commuting power partial isometries.
For more results one can refer to \cite{BDF-MULTI-ISOMETRIES},
\cite{BDF-BI-ISOMETRIES}, \cite{BCL}, \cite{BKPS-COMMUTING}, \cite{BKPS-SHIFT TYPE}, 
\cite{BURDAK-DECOMPOSITION}, \cite{BKS-CANONICAL}, \cite{GG-Wold}, 
\cite{KO-Wold-type}, etc. S\l oci\'{n}ski \cite{SLOCINSKI-MODELS} 
(see also Burdak \cite{BURDAK-DECOMPOSITION}) studied decomposition for 
pairs of doubly commuting contractions and obtained the following result:
\begin{thm}
	Let $T=(T_1, T_2)$ be a pair of doubly commuting contractions on a 
	Hilbert space $\clh$. Then there exists a unique decomposition 
	\[
	\mathcal{H}=\mathcal{H}_{uu} \oplus \mathcal{H}_{u\neg u} 
	\oplus \mathcal{H}_{\neg u u} \oplus \mathcal{H}_{\neg u \neg u}
	\]
	where $\mathcal{H}_{ij}$ are joint $T$-reducing subspaces of $\clh$ 
	for all $i,j=u, \neg u$. Moreover, $T_1$ on $\clh_{ij}$ is unitary 
	if $i=u$ and completely non-unitary if $i=\neg u$ and $T_2$ on 
	$\clh_{ij}$ is unitary if $j=u$ and completely non-unitary if 
	$j=\neg u$. 
\end{thm}
\NI 
However, the complete description of the above orthogonal 
decomposition spaces is not explicit. Burdak \cite{BURDAK-DECOMPOSITION} 
also developed a characterization for pairs of commuting (not necessarily 
doubly commuting) contractions and obtained decomposition results in the case 
of commuting pairs of power partial isometries. 
Recently, Jeu and Pinto \cite{JP-NON COMMUTING} studied a simultaneous 
Wold decomposition for an $n$-tuple $(n > 1)$ of doubly non-commuting 
isometries which has been classified up to unitary equivalence by using 
this decomposition. In 2022, Rakshit, Sarkar, and Suryawanshi 
\cite{RSS-TWISTED ISOMETRIES} showed that each $\clu_n$-twisted 
isometry agrees a von Neumann-Wold type decomposition and then 
described concrete analytic models of $\clu_n$-twisted isometries.
{\it It is now a natural query whether the above results can be extended 
to a large class of operators, namely, a class of twisted contractions
on Hilbert spaces.}

Motivated by the definition of $\clu_n$-twisted isometries in 
\cite{RSS-TWISTED ISOMETRIES}, we introduce the notion of 
$\clu_n$-twisted contractions (see definition in Section 3).
In this paper, we attempt to find a recipe for calculating 
the orthogonal spaces as well as to extend the results for 
pair of doubly twisted contractions (or $\clu_2$-twisted 
contractions) to multi-variable case. Our approach is based 
on the canonical decomposition for a single contraction 
and the geometry of Hilbert spaces.

The paper is organized as follows. In section 2, 
we discuss some basic definitions and the canonical 
decomposition for a single contraction. Section 3
is devoted to the decomposition for a pair of doubly 
twisted contractions. In section 4, we obtain a complete
description for $\clu_n$-twisted contractions and 
in particular for $\clu_n$-twisted isometries.

%%%%%%%%%%%%%%%%%%%%%%%%%%%%%%%%%%%%%%%%%%%%%%%%%%%%%%%%%%%%%%%%%%%%
%%%%%%%%%%%%%%%%%%%%%%%%%%%%%%%%%%%%%%%%%%%%%%%%%%%%%%%%%%%%%%%%%%%%
%%%%%%%%%%%%%%%%%%%%%%%%%%%%%%%%%%%%%%%%%%%%%%%%%%%%%%%%%%%%%%%%%%%%
\newsection {Preparatory Results}
In what follows $\mathbb{Z}_{+}$ denotes the set of non-negative 
integers, $\clh$ stands for a complex Hilbert space, $I$ denotes 
the identity operator on $\clh$ and $\mathcal{B}(\clh)$ 
as the algebra of all bounded linear operators on $\clh$. For a 
closed subspace $\mathcal{M}$ of $\mathcal{H}$, $P_{\clm}$ denotes 
the orthogonal projection of $\clh$ onto $\clm$. A closed subspace 
$\mathcal{M}$ of $\mathcal{H}$ is invariant under $T \in \clb(\clh)$
if $T(\mathcal{M}) \subseteq \mathcal{M}$; and subspace $\mathcal{M}$ 
reduces $T$ if $T(\mathcal{M}) \subseteq \mathcal{M}$ and $ T(\mathcal{M}
^{\perp}) \subseteq \mathcal{M}^{\perp}$. A contraction $T$ on $\clh$ 
(that is, $\|Th \| \leq \|h \|$ for all $h \in \clh$) is said to be a 
pure contraction if $T^{*m} \rightarrow 0$ as $m \rightarrow \infty$
in the strong operator topology. A contraction $T$ on $\mathcal{H}$ 
is called completely non-unitary (c.n.u. for short) if there does 
not exist any nonzero $T$-reducing subspace $\mathcal{L}$ of 
$\mathcal{H}$ such that $T|_{\mathcal{L}}$ is unitary (see \cite{NF-BOOK}). 
We denote $\cln(T)$ and $\clr(T)$ as the kernel and range of $T$, 
respectively. We frequently use the identity $\cln(T)
={\clr(T^*)}^{\perp}$ for any $T \in \clb(\clh)$. Also $\bigvee \clm $ 
stands for the closed linear span of a subset $\clm$  of $\clh$.
An operator $T$ on $\mathcal{H}$ is called a partial isometry if 
$|| Th||= ||h||$ for all $h \in \cln(T)^\perp.$ We say that $T$ is 
a power partial isometry if $T^n$ is a partial isometry for all 
$n \geq 1$.

We now recall \textit{canonical decomposition theorem} for a 
contraction (\cite{NF-BOOK}). In case of an isometry, the canonical
decomposition theorem coincides with the classical Wold-von Neumann
decomposition.

\begin{thm}
A contraction $T$ on a Hilbert space $\mathcal{H}$ corresponds a unique
decomposition of $\mathcal{H}$ into an orthogonal sum of two $T$-reducing 
subspaces $\mathcal{H}={\mathcal{H}}_u \oplus {\mathcal{H}}_{\neg u}$ 
such that $T|_{{\mathcal{H}}_u}$ is unitary and $T|_{{\mathcal{H}}_{\neg u}}$ 
is c.n.u. ( ${\mathcal{H}}_u$ or ${\mathcal{H}}_{\neg u}$ may equal to $\{0\}$).
Moreover,
\[
{\mathcal{H}}_u =	\{ h \in \clh : \|T^n h\|= \|h\|= \|{T^*}^n h\| 
~~\mbox{for}~~ n=1,2, \ldots \}.
\]
Here $T_u=T|_{{\mathcal{H}}_u}$ and $T_{\neg u}=T|_{{\mathcal{H}}_{\neg u}}$ 
are called unitary part and c.n.u. part of $T$, respectively and
$T=T_u \oplus T_{\neg u}$ is called the canonical decomposition of $T$.
\end{thm}

The above theorem can be rewritten as follows:

\begin{thm}\label{Wold-type-contraction}
Let $T$ be a contraction on a Hilbert space $\mathcal{H}$. Then
$\clh$ decomposes as a direct sum of two $T$-reducing subspaces
\begin{align*}
&\clh_u = \bigcap_{m \in \mathbb{Z}_{+}}[ \cln (I-T^{*m}T^m) \cap \cln (I-T^{m}T^{*m})], \\
\mbox{and}\\
& \clh_{\neg u}:= \clh \ominus \clh_u = \bigvee_ {m \in \mathbb{Z}_{+}} \{\clr(I-T^{*m}T^m) 
\cup \clr(I-T^m T^{*m}) \}.
\end{align*}
Also $T_u=T|_{{\mathcal{H}}_u}$ and $T_{\neg u}=T|_{{\mathcal{H}}_{\neg u}}$ 
are called unitary part and c.n.u. part of $T$, respectively.
\end{thm}

\begin{proof}
Suppose that $T \in \clb(\clh)$ is a contraction. Define the defect operators of $T$ as
\[
{D_T}=(I - T^*T)^{1/2}  \ and \  {D_{T^*}}=(I - T T^*)^{1/2}. 
\]
Clearly, ${D_T}$ and ${D_{T^*}}$ are positive operators and bounded by 0 and 1.	
Now for each  $ h \in \clh$ 
\[
\langle D_T^2 h, h \rangle =0 \iff D_T h =0 \iff \| Th \|=\| h \|. 
\]
Therefore, the space $\{ h \in \clh: \|Th\|=\|h\|\}$ coincides with 
$\cln({D_T})=\{h \in \clh: D_Th=0\}$. Consider for each  $m \in \mathbb{Z}$
\begin{align*}
		T(m)=
		\begin{cases}
			T^m  & \text{if} \ m \geq 1,\\
			I    &  \text{if} \ m=0,\\
			T^{*|m|} & \text{if}~~ m \leq -1.
		\end{cases}
\end{align*}
Then for each fixed $m$ in $\mathbb{Z}$, the space 
$\{h \in \clh : \|T(m)h\| = \|h \| \}$ is same as $\cln({D_{T(m)}})
=\{h \in \clh : {D_{T(m)}}h=0\}$. Thus the space $\clh_u$ can be rewritten as 
\begin{align*}
\clh_u 
&=\{ h \in \clh : \|T^m h\|= \|h\|= \|{T^*}^m h\| \ \text{for} \  m \in \mathbb{{N}}\} \\
&=\{h \in \clh: \|T(m)h\|=\|h\| \ \text{for} \ m \in \mathbb{Z}\} \\
&= \bigcap_{m=-\infty}^{\infty} \cln({D_{T(m)}}), 
\end{align*} 
where
\[	
\ D_{T(m)}= 
\begin{cases}
(I-T^{*m}T^m)^{\frac{1}{2}} & \text{if} \ m \geq 0,\\
(I-T^{|m|}T^{*|m|})^{\frac{1}{2}} & \text{if} \ m \leq -1.
\end{cases}
\]	
Since for each $m \in \mathbb{Z}$ the operator $D_{T(m)}$ is positive
on $\clh$, $ \cln (D_{T(m)})= \cln ({D^2_{T(m)}})$. Therefore, 
\begin{align*}
		\cln({D_{T(m)}}) =
		\begin{cases}
			\cln (I-T^{*m}T^m) & \text{if} \ m \geq 0,\\
			\cln (I-T^nT^{*n}) & \text{if} \ n=|m|, \ m \leq -1.
		\end{cases}
\end{align*}
Hence
\begin{align*}
\clh_u &= \bigcap_{m=0}^{\infty} \cln (I-T^{*m}T^m) \cap \bigcap_{n=1}^{\infty}   
\cln (I-T^{n}T^{*n}) \\
&= \bigcap_{m \in \mathbb{Z}_{+}}[ \cln (I-T^{*m}T^m) \cap \cln (I-T^{m}T^{*m})]
\end{align*}
and
\begin{align*}
\clh_{\neg u}:= \clh_u ^{\perp}
&= [\bigcap_{m \in \mathbb{Z}_{+}} \{ \cln (I-T^{*m}T^m) \cap 
\cln (I-T^{m}T^{*m})\} ]^{\perp} \\
&=\bigvee_ {m \in \mathbb{Z}_{+}} \{ \clr(I-T^{*m}T^m) \cup \clr(I-T^{m}T^{*m}) \}.
\end{align*}
	
This completes the proof.	 
\end{proof}

\begin{rem}
If a contraction $T\in \clb(\clh)$ is a power partial isometry, 
then for each $n \in \mathbb{Z}_{+}$, $T^{*n}T^n= P_{\clr(T^{*n})}$ and
$T^nT^{*n}= P_{\clr(T^n)}$, where $P_{\clr(T^{*n})}$ and $P_{\clr(T^n)}$ 
are the orthogonal projections of $\clh$ onto ${\clr(T^{*n})}$ and ${\clr(T^n)}$, 
respectively. Now from Theorem \ref{Wold-type-contraction}, we have  
\begin{align*}
\clh_u &= \bigcap_{n \in \mathbb{Z}_{+}}[\cln(I- T^{*n} T^{n}) \cap \cln(I- T^{n} T^{*n} )] \\
&=\bigcap_{n \in \mathbb{Z}_{+}}[\cln(I- P_{\clr(T^{*n})}) \cap \cln(I-  P_{\clr(T^n) })]\\
&=\bigcap_{n \in \mathbb{Z}_{+}}[\clr( P_{\clr(T^{*n})} ) \cap \clr(P_{\clr(T^{n})} )] \\
&=\bigcap_{n \in \mathbb{Z}_{+}} [T^{*n} \clh \cap  T^{n} \clh],
\end{align*}
and 
\begin{align*}
\clh_{\neg u}  &=\bigvee_ {n \in \mathbb{Z}_{+}} \{ \clr(I-T^{*n}T^n) \cup \clr(I-T^{n}T^{*n}) \} \\
&= \bigvee_ {n \in \mathbb{Z}_{+}} \{ \cln(P_{\clr(T^{*n})}) \cup  \cln(P_{\clr(T^{n})}) \}.
\end{align*}
\end{rem}

\begin{rem}
Let $T$ be an isometry on $\clh$. Then $T^*T=I$ and $T^*\clh = \clh$. 
Since every isometry is a power partial isometry, from the last remark, 
we readily have the unitary part $\clh_u$ of $T$ as
\begin{align*}
\clh_u =  \bigcap_{n \in \mathbb{Z}_{+}} T^n \clh,
\end{align*}
and the c.n.u part $\clh_{\neg u}$ becomes
\begin{align*}
\clh_{\neg u}  =\bigvee_ {n \in \mathbb{Z}_{+}} \{\clr(I-T^{n}T^{*n}) \}
= \bigvee_ {n \in \mathbb{Z}_{+}} \{  \cln(P_{\clr(T^{n})}) \}.	
\end{align*}
Again for $n \geq 1$ 
\begin{align*}\label{equation 0}
\clr(I-T^{n}T^{*n})&= \clh \ominus T^{n}T^{*n}\clh 
= \clh \ominus T^{n}\clh  \nonumber \\
&= (\clh \ominus T\clh) \oplus (T\clh \ominus T^2\clh) 
\oplus \cdots \oplus (T^{n-1} \clh \ominus T^n \clh) \nonumber \\
&= \cln ( T^*) \oplus T \cln ( T^*) \oplus \cdots \oplus T^{n-1}\cln ( T^*) \nonumber\\
&= \bigoplus_{k=0}^{n-1} T^k \cln ( T^*).
\end{align*}
Since $ \clr(I-TT^*) \subseteq \clr(I-T^2 T^{*2}) \subseteq \cdots \subseteq 
\clr(I-T^n T^{*n}) \subseteq \cdots $, 
\begin{align*}
\clh_{\neg u} =\bigvee_{n=1}^{\infty} \Big\{ \clr(I-T^{n}T^{*n}) \Big\}
= \bigoplus_{n=0}^{\infty}  T^n \cln (T^*).
\end{align*}
Therefore, the canonical decomposition of $T$ coincides 
with the {\it Wold decomposition}.
\end{rem}

%%%%%%%%%%%%%%%%%%%%%%%%%%%%%%%%%%%%%%%%%%%%%%%%%%%%%%%%%%%%%%%%%%%%%%%%%
%%%%%%%%%%%%%%%%%%%%%%%%%%%%%%%%%%%%%%%%%%%%%%%%%%%%%%%%%%%%%%%%%%%%%%%%%
%%%%%%%%%%%%%%%%%%%%%%%%%%%%%%%%%%%%%%%%%%%%%%%%%%%%%%%%%%%%%%%%%%%%%%%%%
\section{Decomposition for $\clu_2$-twisted contractions}

In this section, we achieve the explicit orthogonal 
decomposition spaces for pairs of doubly twisted contractions 
(in particular, doubly twisted isometries) on Hilbert 
spaces. Our approach is different and the results unify 
all the existing results in the literature studied by 
many researchers, like S\l oci\'{n}ski \cite{SLOCINSKI-WOLD}, 
Burdak \cite{BURDAK-DECOMPOSITION}, Popovici \cite{POPOVICI-WOLD TYPE}, 
\cite{POPOVICI-On}, Catepill\'{a}n et al. \cite{CPS-Multiple}, 
and the recent results of Jeu and Pinto \cite{JP-NON COMMUTING},
Rakshit, Sarkar, and Suryawanshi \cite{RSS-TWISTED ISOMETRIES}. 

We shall work in the following fixed set-up.

\begin{defn}($\clu_n$-twisted contractions)
Let $n>1$ and $\{U_{ij}\}$ for $1 \leq i < j \leq n$ be $\binom{n}{2}$ 
commuting unitaries on a Hilbert space $\clh$ such that $U_{ji}:=U^*_{ij}$.
We say that an $n$-tuple of contractions $(T_1, \dots, T_n)$ on $\clh$ 
is a $\clu_n$-twisted contraction with respect to a twist $\{U_{ij}\}_{i<j}$ 
if 
\[
T_iT_j=U_{ij}T_jT_i; \hspace{1cm}\hspace{1cm} T_i^*T_j= U^*_{ij}T_jT_i^* 
\hspace{1cm} \mbox{and} \hspace{1cm} T_kU_{ij} =U_{ij}T_k
\]
for all $i,j,k=1, \dots, n$ and $i \neq j$. We simply say that the tuple 
$(T_1, \dots, T_n )$ is a $\clu_n$-twisted contractions without 
referencing the twist $\{U_{ij}\}_{1 \leq i<j \leq n}$.
\end{defn}
\NI
In particular, if $U_{ij}=I$ for all $1 \leq i < j \leq n$, 
then the tuple $(T_1, \ldots, T_n)$ is said to be doubly 
commuting contraction, that is, $T_iT_j= T_j T_i$ and 
$T_iT_j^* = T_j^*T_i$ for $1\leq i<j\leq n$. If $n=2$, 
then we shall refer to $(T_1, T_2)$ as a pair of 
\textit{doubly twisted contraction} or \textit{$\clu_2$-twisted 
contraction} on $\clh$. If $n=1$, then $n$-tuple reduces 
to a single contraction.

\begin{rem}
Let $(T_1, \dots, T_n)$ be an $n$-tuple of isometries on $\clh$.  
For $1 \leq i < j \leq n$, let  $\{U_{ij}\}$ be $\binom{n}{2}$ 
commuting unitaries on a Hilbert space $\clh$ such that 
$U_{ji}:=U^*_{ij}$. Then the relation 
$T_i^*T_j=U^*_{ij} T_j T_i^*$ and $T_kU_{ij} =U_{ij}T_k$
implies $T_iT_j={U}_{ij} T_j T_i $ for all $i,j,k=1, \dots, n$ 
and $i \neq j$ (see \cite{JPS-q COMMUTING ISOMETRIES}, \cite{RSS-TWISTED ISOMETRIES}). 
However, this fact is not true for an $n$-tuple of contractions.  
\end{rem}

The following result is simple, but plays an important role in the sequel.
\begin{lem}\label{commuting operators-cor1}
Let $(T_1, \dots, T_n)$ be an $n$-tuple of $\clu_n$-twisted contractions 
on a Hilbert space $\clh$, and let $l,m \in \mathbb{Z}_{+}$. Then for 
all $i \neq j$,
\begin{enumerate}
\item $T_i$ commutes with $T_j^mT_j^{*m}$ and $T_j^{*m}T_j^{m}$;
\item $T_i^*$ commutes with   $T_j^mT_j^{*m}$ and  $T_j^{*m}T_j^{m}$;
\item $T_i^lT_i^{*l}$, $T_i^{*l} T_i^l$ commute with the operators 
	 $(I-T_j^mT_j^{*m})$ and $(I -T_j^{*m}T_j^{m})$.
\end{enumerate}
\end{lem}

\begin{proof}
Suppose that $l,m \in \mathbb{Z}_{+}$. Now for $i \neq j$, using the 
definition repeated times, we have
\begin{align*}
& T_iT_j^mT_j^{*m}= U_{ij}^m T_j^m T_i T_j^{*m}= U_{ij}^m {U}_{ij}^{*m} 
T_j^m  T_j^{*m}T_i= T_j^m  T_j^{*m}T_i  \\
\mbox{and} \hspace{0.5cm}
& T_iT_j^{*m}T_j^m= {U}_{ij}^{*m}   T_j^{*m}T_i T_j^m	= {U}_{ij}^{*m} 
U_{ij}^m  T_j^{*m} T_j^m T_i =T_j^{*m} T_j^m T_i.
\end{align*}
Hence the first part is proved. 
	
Second part follows from the first part by just taking the adjoint 
of those operators. Using the part (1) and (2), we can easily prove 
the last part.	
\end{proof}

We shall first concentrate on pairs of \textit{doubly twisted contractions} 
with some examples on Hilbert spaces $\clh$ and their decomposition.

\begin{ex}\label{Example-doubly non commuting}
Let $H^2(\mathbb{D})$ denotes as the Hardy space over the unit disc 
$\mathbb{D}$. The weighted shift $M_z^{\alpha}$ on $H^2(\mathbb{D})$ is 
defined by $M_z^{\alpha}(f)= \alpha zf$ for all $f \in  H^2(\mathbb{D})$, 
where $z$ is the co-ordinate function and $|\alpha| \leq 1$. Now the Hardy 
space over the bidisc $\mathbb{D}^2$, denoted by $H^2(\mathbb{D}^2)$, can be 
identified with $H^2(\mathbb{D}) \otimes  H^2(\mathbb{D})$ through the canonical 
unitary $\Gamma:  H^2(\mathbb{D}) \otimes H^2(\mathbb{D}) \rightarrow H^2(\mathbb{D}^2)$
defined by $\Gamma(z^{m_1} \otimes z^{m_2})= z_1^{m_1} z_2^{m_2}$ for 
$(m_1, m_2) \in \mathbb{Z}_+^2$. 
	
For each fixed $ r \in S^1:=\{z \in \mathbb{C}: |z|=1 \},$ we define 
an operator $A_r$ on $H^2(\mathbb{D})$ as
\[
A_r z^n= \frac{r^n}{2} z^n \hspace{1cm}  (n \in \mathbb{Z}_+), 
\]
where $\{1,z,z^2, \dots \}$ is an orthonormal basis for $H^2(\mathbb{D})$. 
Then 
\begin{equation*}
( M_z^{\alpha} A_r)(z ^n)=  \frac{\alpha r^n}{2} z^{n+1} \quad \text{and} \quad
( A_r M_z^{\alpha})(z^n) = \frac{\alpha r^{n+1}}{2} z^{n+1} ~~ \text{for}~~ n \in \mathbb{Z}_+.
\end{equation*}
Again
\begin{equation*}
\big([M_z^{\alpha}]^* A_r\big)(z ^n)=
\begin{cases}
\frac{\bar{\alpha} r^n}{2} z^{n-1}, & \text{if $n \geq 1$}\\
0 & \text{if $n = 0$},
\end{cases}
\end{equation*}
and 
\begin{equation*}
		\big(A_r [M_z^{\alpha}]^*\big)(z ^n)=
		\begin{cases}
			\frac{\bar{\alpha} r^{n-1}}{2} z^{n-1} & \text{if $n \geq 1$}\\
			0 & \text{if $n = 0$}.
		\end{cases}
\end{equation*}
Hence $ A_r M_z^{\alpha}=rM_z^{\alpha}A_r$ and    
$ [M_z^{\alpha}]^*A_r= rA_r [M_z^{\alpha}]^*$. 
We now define $T_1$ and $T_2$ on $H^2(\mathbb{D}^2)$ such that 
\[
T_1=  A_r \otimes  M_z^{\alpha}
	\ \ \ \text{and} \quad 
T_2=  M_z^{\alpha} \otimes I_{H^2(\mathbb{D})}. 
\] 
Therefore, we can check that $(T_1,T_2)$ is a pair of contractions on 
$H^2(\mathbb{D}^2)$. Moreover,
\[
T_1T_2=  A_r M_z^{\alpha} \otimes M_z^{\alpha}=  r M_z^{\alpha} A_r 
\otimes M_z^{\alpha}= r(  M_z^{\alpha} A_r \otimes  M_z^{\alpha})= rT_2T_1. 
\]
and
\[
T_2^*T_1= [M_z^{\alpha}]^*A_r \otimes  M_z^{\alpha} = r A_r [M_z^{\alpha}]^* 
\otimes M_z^{\alpha} = r(A_r [M_z^{\alpha}]^* \otimes M_z^{\alpha})=rT_1 T_2^*. 
\]
Consider $\clh=H^2(\mathbb{D}^2) \oplus H^2(\mathbb{D}^2)$. We now define 
two contractions on $\clh$ as $T_1'=\diag(T_1,T_2)$ and $T_2'=\diag(T_2,T_1)$. 
Set $U=\diag(rI_{H^2(\mathbb{D}^2)},\bar{r} I_{H^2(\mathbb{D}^2)})$, $|r|=1$. 
Clearly, $U$ is unitary on $\clh$ and 
\[
	T_1'T_2'=\begin{bmatrix}
		T_1T_2 & 0\\
		0 & T_2T_1
	\end{bmatrix}
	=\begin{bmatrix}
		rT_2T_1 & 0\\
		0 & \bar{r}T_1T_2
	\end{bmatrix}
	=\begin{bmatrix}
		rI_{\clh^2(\mathbb{D}^2)} & 0\\
		0 & \bar{r}I_{\clh^2(\mathbb{D}^2)}
	\end{bmatrix} T_2'T_1'=U T_2'T_1'
\]
and
\[
	T_2'^*T_1'
	=\begin{bmatrix}
		T_2^*T_1 & 0\\
		0 & T_1^*T_2
	\end{bmatrix}
	=\begin{bmatrix}
		rT_1T_2^* & 0\\
		0 & \bar{r}T_2T_1^*
	\end{bmatrix}
	=\begin{bmatrix}
		rI_{\clh^2(\mathbb{D}^2)} & 0\\
		0 & \bar{r}I_{\clh^2(\mathbb{D}^2)}
	\end{bmatrix} T_1'T_2'^*= UT_1'T_2'^*.
\]
Again $T_1'U =UT_1'$, and $T_2'U = UT_2'$. So it follows that $(T_1',T_2')$ 
is a $\clu_2$-twisted contractions on $\clh$ with a twist $\clu_2=\{U\}$.
\end{ex}

\begin{ex}
Let $H^2_{\cle}(\mathbb{D}^2)$ denotes as the $\cle$-valued Hardy space 
over the unit bidisc $\mathbb{D}^2$, where $\cle$ is any Hilbert space. 
We can also identify $H^2_{\cle}(\mathbb{D}^2)$ as $H^2(\mathbb{D}^2) \otimes \cle$. 
The weighted shift operators $M_{z_i}^{\alpha_i}$ is defined by
$M_{z_i}^{\alpha_i} f= \alpha_i z_i f$  for $f \in H^2_{\cle}(\mathbb{D}^2)$, 
where $z_i \in \mathbb{D}$, $|\alpha_i|\leq 1$ for $i=1,2.$  
We now define operators $T_1$ and $T_2$ on $H^2_{\cle}(\mathbb{D}^2)$ as 
\[
T_1=M_{z_1}^{\alpha_1} \hspace{1cm} \text{and}  \hspace{1cm}
T_2=M_{z_2}^{\alpha_2} D[U],
\]
where $U$ is unitary on $\cle$ and $D[U]$ on $H^2_{\cle}(\mathbb{D}^2)$ is defined by 	
\[
D[U] (z_1^{m_1}z_2^{m_2} \eta)= z_1^{m_1} z_2^{m_2}(U^{m_1} \eta) \quad \mbox{for} 
\quad (m_1,m_2) \in \mathbb{Z}_+^2, \ \eta \in  \cle.
\] 
It is easy to check that $(T_1,T_2)$ is a pair of contractions on 
$H^2_{\cle}(\mathbb{D}^2)$. Moreover,
\begin{align*}
	T_2T_1(z_1^{m_1}z_2^{m_2} \eta)
	=M_{z_2}^{\alpha_2} D[U](\alpha_1 z_1^{m_1+1}z_2^{m_2} \eta)
	=\alpha_1 \alpha_2 z_1^{m_1+1}z_2^{m_2+1} U^{m_1+1}\eta,
\end{align*}
and
\begin{align*}
   	T_1T_2(z_1^{m_1}z_2^{m_2} \eta)
   	&=M_{z_1}^{\alpha_1} (\alpha_2 z_1^{m_1}z_2^{m_2+1}U^{m_1} \eta)\\
   	&=\alpha_1 \alpha_2 z_1^{m_1+1}z_2^{m_2+1}U^{m_1} \eta \\
   	&=(I_{H^2({\mathbb{D}})} \otimes U^*) (\alpha_1 \alpha_2 
   	z_1^{m_1+1}z_2^{m_2+1} U^{m_1+1}\eta)\\
   	&=\widetilde{U} T_2T_1 (z_1^{m_1}z_2^{m_2} \eta)
\end{align*}
where $\widetilde{U}=(I_{H^2({\mathbb{D}^2})} \otimes U^*)$ is unitary 
on $H^2_{\cle}(\mathbb{D}^2)$ (or on $H^2(\mathbb{D}^2) \otimes {\cle}$). 
Therefore 
\[
T_1T_2=\widetilde{U} T_2T_1
\]
on $H^2_{\cle}(\mathbb{D}^2)$. Again 
\begin{align*}
T_1^*T_2(z_1^{m_1} z_2^{m_2}\eta)
&=[M_{z_1}^{\alpha_1}]^*( \alpha_2 z_1^{m_1} z_2^{m_2+1} U^{m_1}\eta) \\
&=\begin{cases}
\bar{\alpha_1} \alpha_2 z_1^{m_1-1} z_2^{m_2+1} U^{m_1}\eta  & \text{if $m_1 \geq 1$}\\
0 & \text{if $m_1 = 0$},
\end{cases}
\end{align*}
and 
\begin{align*}
T_2T_1^*(z_1^{m_1} z_2^{m_2}\eta)
&= \begin{cases}
  M_{z_2}^{\alpha_2} D[U] (\bar{\alpha_1}z_1^{m_1-1} z_2^{m_2}\eta) &
   \text{if $m_1 \geq 1$}\\
  0 & \text{if $m_1 = 0$}
\end{cases} \\
&=\begin{cases}
	\bar{\alpha_1} \alpha_2 z_1^{m_1-1} z_2^{m_2+1} U^{m_1-1}\eta  & \text{if $m_1 \geq 1$}\\
	0 & \text{if $m_1 = 0$}.
\end{cases}
\end{align*}	
Therefore, $T_1^*T_2(z_1^{m_1} z_2^{m_2}\eta)= (I_{H^2({\mathbb{D}^2})} 
 \otimes U)T_2T_1^*(z_1^{m_1} z_2^{m_2}\eta)$, that is,
$T_1^*T_2=\widetilde{U}^*T_2T_1^*$ on $H^2_{\cle}(\mathbb{D}^2)$. Moreover, 
$T_i \widetilde{U}=\widetilde{U} T_i$ for $i=1,2$. Hence $(T_1,T_2)$ is 
a $\clu_2$-twisted contraction on $H^2_{\cle}(\mathbb{D}^2)$.

\NI
In particular, we take $\cle= \ell^2(\mathbb{Z})$ and the bilateral shift 
$W$ on $\ell^2(\mathbb{Z})$. Then the above pair $(T_1,T_2)$ is 
a {\it{doubly twisted contraction}} with respect to the twist 
$\{I_{H^2({\mathbb{D}^2})} \otimes W^*\}$ on the Hilbert space 
$H^2_{ \ell^2(\mathbb{Z}) }(\mathbb{D}^2)$.
\end{ex}

\begin{ex}
For each fixed $ r \in S^1,$ we define a weighted shift operator
$B_r$ on the Hardy space $H^2(\mathbb{D})$ such that
\[
B_r z^n= r^{n+1} z^{n+1}\hspace{1cm}  (n \in \mathbb{Z}_+) 
\]
where $\{1,z,z^2, \dots \}$ is an orthonormal basis for $H^2(\mathbb{D})$. 
Let $M_z$ be the multiplication operator on $H^2(\mathbb{D})$ by the 
coordinate function $z$. Then 
\begin{equation*}
		( M_z B_r)(z ^n)= r^{n+1} z^{n+2} 
		\quad \text{and} \quad
		(  B_r M_z)(z^n)= r^{n+2} z^{n+2}
		\ \ \text{for  $n \in \mathbb{Z}_+$}.
\end{equation*}
Again
\begin{equation*}
( M_z^* B_r)(z ^n)= r^{n+1} z^{n} ~~~ \text{ $\forall \ n \geq 0$},\\
\end{equation*}
and 
\begin{equation*}
(B_r M_z^*)(z ^n)=
\begin{cases}
r^{n} z^{n} & \text{if $n \geq 1$}\\
0 & \text{if $n = 0$}.
\end{cases}
\end{equation*}
Hence $B_r M_z=rM_zB_r$ but $ M_z^*B_r \neq  rB_r M_z^*$. 
Now define $T_1,T_2$ on $H^2(\mathbb{D}^2)$ as
\[
	T_1=  B_r \otimes  M_z
	\ \ \ \text{and} \quad 
	T_2=  M_z \otimes I_{H^2(\mathbb{D})}. 
\] 
Then it is easy to see that $(T_1, T_2)$ is a pair of isometries on 
$H^2(\mathbb{D}^2)$. Also
\[
T_1T_2= B_rM_z \otimes M_z= r M_z B_r \otimes M_z= r(M_zB_r \otimes M_z)=rT_2T_1.
\]
On the other hand
\[
	T_2^*T_1= M_z^*B_r \otimes  M_z 
	\quad \text{and} \quad 
	T_1 T_2^*= B_r M_z^* \otimes M_z. 
\]
Therefore, $T_2^*T_1 \neq  rT_1T_2^*$ as $ M_z^*B_r \neq rB_r M_z^*$. 	
Consider $\clh=H^2(\mathbb{D}^2) \oplus H^2(\mathbb{D}^2)$. We now define 
two isometries on $\clh$ as $T_1'=\diag(T_1,T_2)$ and $T_2'=\diag(T_2,T_1)$. 
Set $U=\diag(rI_{H^2(\mathbb{D}^2)},\bar{r} I_{H^2(\mathbb{D}^2)})$, $|r|=1$. 
Clearly, $U$ is unitary on $\clh$ and 
\[
T_1'T_2'=U T_2'T_1', \quad T_1'U =UT_1',  \quad T_2'U = UT_2'
\]
but $T_2'^*T_1'\neq UT_1'T_2'^*$. Therefore, $(T_1',T_2')$ is not 
a pair of {\it{doubly twisted isometry}} on $\clh$ with a twist $\clu_2=\{U\}$.
\end{ex}

%%%%%%%%%%%%%%%%%%%%%%%%%%%%%%%%%%%%%%%%%%%%%%%%%%%%%%%%%%%%%%%%%

The following result will be used frequently in the sequel.
\begin{lem}\label{Reduce-subspace-prop1}
Let $(T_1, T_2)$ be a pair of {\it{doubly twisted}} operator on $\clh$ 
such that $T_1$ is a contraction. Let $\clh=\clh_u^1 \oplus \clh_{\neg u}^1$ 
be the canonical decomposition of contraction $T_1$. Then the decomposition
reduces $T_2$. 
\end{lem}

\begin{proof}
Suppose that $(T_1, T_2)$ is a pair of {\it{doubly twisted}} operator 
with a twist $\clu_2=\{U\}$ on $\clh$ and $T_1$ is a contraction. 
Then from the above Theorem \ref{Wold-type-contraction}, we get   
$\clh= \clh_u^1  \oplus \clh_{\neg u}^1$, where $\clh_u^1, \clh_{\neg u}^1$ 
reduce $T_1$ and
\begin{align*}
		\clh_u^1 & = \bigcap_{m \in \mathbb{Z}_{+}}[ \cln (I-T_1^{*m}T_1^m) \cap 
		\cln (I-T_1^{m}T_1^{*m})], \\ 
		\clh_{\neg u}^1 & = \bigvee_ {m \in \mathbb{Z}_{+}} \{ \clr(I-T_1^{*m}T_1^m) 
		\cup \clr(I-T_1^m T_1^{*m}) \}.
\end{align*}
	Since $(T_1, T_2)$ is a pair of doubly twisted operator, using Lemma 
	\ref{commuting operators-cor1} we have $T_2(\clh_u^1)\subseteq 
	\clh_u^1$ and $T_2(\clh_{\neg u}^1)\subseteq \clh_{\neg u}^1$. 
	
	This finishes the proof.
\end{proof}

Let $\left( T_1, T_2 \right)$ be a pair of {\it{doubly twisted contraction}}
on $\clh$. Suppose that $\clh=\clh_u^1 \oplus \clh_{\neg u}^1$ 
is the canonical decomposition for contraction $T_1$ such that $T_1|_{\clh_u^1}$ 
is unitary and $T_1|_{\clh_{\neg u}^1}$ is completely non-unitary. 
The above Lemma \ref{Reduce-subspace-prop1} implies that
the subspaces $\clh_u^1$, $\clh_{\neg u}^1$ reduce the contraction $T_2$ 
and from Theorem \ref{Wold-type-contraction}, we get  
\begin{align*}
&\clh_u^1= \bigcap_{ m_1 \in \mathbb{Z}_{+}}[\cln(I- T_1^{* m_1} T_1^{ m_1}) 
	\cap \cln(I- T_1^{ m_1} T_1^{* m_1} )], \\
& \clh_{\neg u}^1= \bigvee_{m_1 \in \mathbb{Z}_{+}} \{ \clr(I-T_1^{* m_1}T_1^{ m_1})
	\cup \clr(I-T_1^{ m_1} T_1^{* m_1})
	\}.
\end{align*}
Since $T_2|_{\clh_u^1}$ is a contraction, the canonical decomposition yields 
$\clh_u^1=\clh_{uu} \oplus \clh_{u\neg u}$, where $T_1|_{\clh_{uu}}$, 
$T_1|_{\clh_{u\neg u}}$, $T_2|_{\clh_{uu}}$ are unitary and 
$ T_2|_{\clh_{u\neg u}}$ is c.n.u. 
Again $\clh_u^1$ reduces $T_2$, and hence $(T_2|_{\clh_u^1})^m= T_2^m|_{\clh_u^1}$ 
and $(T_2|_{\clh_u^1})^{*m} = T_2^{*m}|_{\clh_u^1}$ for any $m \in \mathbb{Z}_{+}$. 
Thus from Theorem \ref{Wold-type-contraction}, the subspaces $\clh_{uu}$ 
and $\clh_{u\neg u}$ can be written as
\begin{align*}
	\clh_{uu}
	& = \bigcap_{ m_2 \in \mathbb{Z}_{+}} [\cln (I_{\clh_u^1}-T_2^{* m_2}|_{\clh_u^1}  T_2^{ m_2}|_{\clh_u^1}) 
	\cap \cln (I_{\clh_u^1}-T_2^{ m_2}|_{\clh_u^1} T_2^{* m_2}|_{\clh_u^1})] \\
	& = \bigcap_{m_2 \in \mathbb{Z}_{+}} [\cln ((I -T_2^{* m_2} T_2^{ m_2})|_{\clh_u^1}) 
	\cap \cln ((I -T_2^{ m_2} T_2^{* m_2})|_{\clh_u^1})],\\
	\mbox{and}\\
	\clh_{u\neg u} & = \bigvee_{ m_2 \in \mathbb{Z}_{+}} \{ (I-T_2^{* m_2}T_2^{ m_2}) \clh_u^1 \cup
	(I-T_2^{ m_2}T_2^{* m_2}){\clh_u^1} \}.
\end{align*}
For the rest of the part, let $\clh_{\neg u}^1=\clh_{\neg u u} \oplus \clh_{\neg u \neg u}$ 
be the canonical decomposition for contraction $T_2|_{\clh_{\neg u}^1}$ 
such that $T_2|_{\clh_{\neg uu}}$ is unitary and $T_1|_{\clh_{\neg uu}}, 
T_1|_{\clh_{\neg u\neg u}}, T_2|_{\clh_{\neg u \neg u}}$ are c.n.u.
Since $\clh_{\neg u}^1$ is a $T_2$-reducing subspace, $\clh_{\neg u}^1$ reduces 
$(I-T_2^{* m_2}T_2^{ m_2})$ and $(I-T_2^{ m_2}T_2^{* m_2})$. Therefore 
from Theorem \ref{Wold-type-contraction}, the subspaces $\clh_{\neg u u}$ 
and $ \clh_{\neg u \neg u}$ can be written as 
\begin{align*}
\clh_{\neg u u} 
& = \bigcap_{ m_2 \in \mathbb{Z}_{+}} [\cln ( (I-T_2^{*m_2} T_2^{ m_2})|_{\clh_{\neg u }^1})
\cap \cln ((I_{\clh_{\neg u }^1}-T_2^{ m_2} T_2^{* m_2}) |_{\clh_{\neg u }^1})],\\ 
\mbox{and}\\
\clh_{\neg u\neg u} & = \bigvee_{ m_2 \in \mathbb{Z}_{+}} \{ (I-T_2^{* m_2}T_2^{ m_2}) \clh_{\neg u}^1 \cup 
(I-T_2^{ m_2}T_2^{* m_2}){\clh_{\neg u}^1} \}.
\end{align*}

To summarize the above, we have the following result:   
\begin{thm}\label{Wold-type-pair-contractions}
Let $\left( T_1, T_2 \right)$ be a pair of {\it{doubly twisted contraction}} 
on a Hilbert space $\clh$. Then there is a unique decomposition 
\[
\mathcal{H}=\mathcal{H}_{uu} \oplus \mathcal{H}_{u\neg u} \oplus 
\mathcal{H}_{\neg u u} \oplus \mathcal{H}_{\neg u \neg u},
\]
where $\mathcal{H}_{uu},  \mathcal{H}_{u\neg u}, \mathcal{H}_{\neg u u}$, and
$\mathcal{H}_{\neg u \neg u}$ are the subspaces reduce $T_1$ and $T_2$ such that
\begin{itemize}
\item	$ T_1|_{\mathcal{H}_{uu}},T_2|_{\mathcal{H}_{uu}}$ 
		are unitary,		
\item   $T_1|_{\mathcal{H}_{u\neg u}}$ is unitary and $T_2|_{\mathcal{H}_{u \neg u}}$ is c.n.u.,
		
\item   $T_1|_{\mathcal{H}_{\neg u u}}$ is c.n.u. and $T_2|_{\mathcal{H}_{\neg u u}}$ is unitary,
	
\item   $T_1|_{\mathcal{H}_{\neg u\neg u}}, T_2|_{\mathcal{H}_{\neg u\neg u}}$ are c.n.u. 
\end{itemize} 
Moreover, the orthogonal subspaces can be formulated as
\begin{align*}
&\clh_{uu}
=\bigcap_{ m_2 \in \mathbb{Z}_{+}} [\cln ((I -T_2^{* m_2} T_2^{ m_2})|_{\clh_u^1}) 
		\cap \cln ((I -T_2^{ m_2} T_2^{* m_2})|_{\clh_u^1})],\\
& \clh_{u\neg u}= \bigvee_{  m_2 \in \mathbb{Z}_{+}} \{ (I-T_2^{* m_2}T_2^{ m_2}) \clh_u^1 \cup 
		(I-T_2^{ m_2}T_2^{* m_2}){\clh_u^1} \}, \\
&\clh_{\neg u u}=
		\bigcap_{ m_2 \in \mathbb{Z}_{+}} [\cln ( (I-T_2^{*m_2} T_2^{ m_2})|_{\clh_{\neg u }^1})
		\cap \cln ((I_{\clh_{\neg u }^1}-T_2^{ m_2} T_2^{* m_2}) |_{\clh_{\neg u }^1})], \\
& \clh_{\neg u\neg u}= \bigvee_{ m_2 \in \mathbb{Z}_{+}} \{ (I-T_2^{* m_2}T_2^{ m_2}) \clh_{\neg u}^1 \cup 
		(I-T_2^{ m_2}T_2^{* m_2}){\clh_{\neg u}^1} \}, 
\end{align*}
and
\begin{align*}
		&\clh_u^1= \bigcap_{ m_1 \in \mathbb{Z}_{+}}[\cln(I- T_1^{* m_1} T_1^{ m_1}) 
		\cap \cln(I- T_1^{ m_1} T_1^{* m_1} )], \\
		& \clh_{\neg u}^1= \bigvee_{  m_1 \in \mathbb{Z}_{+}} \{ \clr(I-T_1^{* m_1}T_1^{ m_1})
		\cup \clr(I-T_1^{ m_1} T_1^{* m_1})
		\}.
\end{align*}
\end{thm}

Return to our discussion on the decomposition for pairs of {\it{doubly 
twisted isometries}}. Suppose that $(T_1,T_2)$ is a pair of 
{\it{doubly twisted isometry}} with a twist $\clu_2=\{U\}$ on $\clh$. 
Then from the above Theorem \ref{Wold-type-pair-contractions}, we get
\[
\clh_u^1=  \bigcap_{ m_1=0}^{\infty} T_1^{ m_1} \clh \ \ \text{and} \ \
\clh_{\neg u}^1= \bigoplus_{ m_1=0}^{\infty}  T_1^{ m_1}\cln ( T_1^*).
\]
Again
\begin{align*}
	\clh_{uu} 
	&=\bigcap_{ m_2 \in \mathbb{Z}_{+}} \cln (I_{\clh_u^1}-T_2^{ m_2}
	|_{\clh_u^1} T_2^{* m_2}|_{\clh_u^1})
	=\bigcap_{ m_2 \in \mathbb{Z}_{+}} \big[\cln (T_2^{ *m_2}|_{\clh_u^1} ) \big] ^{\perp}\\
	&= \bigcap_{ m_2 \in \mathbb{Z}_{+}} T_2^{ m_2} \clh_u^1 
	= \bigcap_{ m_1, m_2 \in \mathbb{Z}_{+}} T_1^{ m_1} T_2^{ m_2} \clh,\\
	\mbox{and}\\
	\clh_{u\neg u} 
	&= \bigvee \Big\{  (I-T_2^{ m_2}T_2^{* m_2}){\clh_u^1}: m_2 \in \mathbb{Z}_{+} \Big\} 
	=\bigvee_{ m_2 \in \mathbb{Z}_{+}}  \bigg\{ (I-T_2^{ m_2}T_2^{* m_2})
	\big[\bigcap_{ m_1=0}^{\infty} T_1^{ m_1} \clh \big]\bigg\}.
\end{align*}
Since $T_2$ is an isometry and $\clh_u^1$ reduces $T_2$, 
$(I-T_2^{ m_2}T_2^{* m_2})|_{\clh_u^1}$ is a projection
for any fixed $m_2 \in \mathbb{Z}_{+}$. As $(T_1, T_2)$
is a pair of {\it{doubly twisted isometry}} and using 
Lemma \ref{commuting operators-cor1}, we get
\[
(I-T_2^{ m_2}T_2^{* m_2})T_1= T_1(I-T_2^{ m_2}T_2^{* m_2})
\iff P_{\cln(T_2^{*m_2})}T_1=T_1 P_{\cln(T_2^{*m_2})}.
\]
Again $\cln(T_2^{*})$ reduces the unitary $U$ as the pair $(T_2, U)$
is doubly commuting. 

Hence for any fixed $m_2 \geq 1$, we have
\begin{align*}
(I-T_2^{ m_2}T_2^{* m_2}) [\bigcap_{ m_1=0}^{\infty} T_1^{ m_1} \clh ]
&= (I-T_2^{ m_2}T_2^{* m_2}) \clh \bigcap  (I-T_2^{ m_2}T_2^{* m_2})T_1 
\clh \bigcap \ldots \qquad \\
&=  (\clh \ominus T_2^{ m_2}T_2^{* m_2} \clh) \bigcap T_1( \clh \ominus 
T_2^{ m_2}T_2^{* m_2} \clh) \bigcap \ldots\\
&=  \bigg(\bigoplus_{k=0}^{ m_2-1} T_2^k \cln (T_2^*)\bigg) \bigcap T_1
\bigg(\bigoplus_{k=0}^{ m_2-1} T_2^k \cln (T_2^*)\bigg) \bigcap \ldots \\
&=  \bigg(\bigoplus_{k=0}^{ m_2-1} T_2^k \cln (T_2^*)\bigg) \bigcap \bigg
(\bigoplus_{k=0}^{ m_2-1} T_2^k T_1 \cln( T_2^*)\bigg) \bigcap \ldots \\
&=\bigcap_{ m_1=0}^{\infty} \bigg(\bigoplus_{k=0}^{ m_2-1} T_2^k T_1^{ m_1}\cln( T_2^*)\bigg) \\
&=\bigoplus_{k=0}^{ m_2-1} T_2^k \bigg(\bigcap_{ m_1=0}^{\infty} T_1^{ m_1}\cln( T_2^*)\bigg). 
\end{align*}
Hence
\begin{align*}
\clh_{u \neg u}&=\bigoplus_{ m_2=0}^{\infty} T_2^{ m_2} \bigg(\bigcap_{ m_1=0}
^{\infty} T_1^{ m_1} \cln (T_2^*)\bigg). 
\end{align*}
Again
\begin{align*}
\clh_{\neg u u}
&= \bigcap_{ m_2 \in \mathbb{Z}_{+}} \cln ((I - T_2^{ m_2}T_2^{* m_2})|_{\clh_{\neg u }^1}) \\
&= \bigcap_{ m_2 \in \mathbb{Z}_{+}} T_2^{ m_2} {\clh_{\neg u }^1} \\
&= \bigcap_{m_2 \in \mathbb{Z}_{+}} \bigg( \bigoplus_{ m_1=0}^{\infty}  
T_1^{ m_1} T_2^{ m_2}\cln (T_1^*)\bigg) \\
& =\bigoplus_{ m_1=0}^{\infty}  T_1^{ m_1} \bigg(\bigcap_{ m_2 \in \mathbb{Z}_{+}} 
T_2^{ m_2} \cln( T_1^*) \bigg).
\end{align*}

Finally by Lemma \ref{commuting operators-cor1}, we have
\begin{align*}
\clh_{\neg u\neg u } &= \bigvee \{(I-T_2^{ m_2}T_2^{* m_2}){\clh_{\neg u}^1}:  m_2 \in \mathbb{Z}_{+} \} \\
&= \bigvee_{  m_2 \in \mathbb{Z}_{+}} \bigg\{  (I-T_2^{ m_2}T_2^{* m_2})
\Big(\bigoplus_{ m_1=0}^{\infty}  T_1^{ m_1}\cln( T_1^*) \Big)\bigg\} \\
&=  \bigvee_{  m_2 \in \mathbb{Z}_{+}} \bigg\{ 
\bigoplus_{ m_1=0}^{\infty}  T_1^{ m_1}(I-T_2^{ m_2}T_2^{* m_2}) \big[ \cln (T_1^*) \big] \bigg\}.
\end{align*}
Since $(T_1, T_2)$ is a pair of {\it{doubly twisted isometry}}, we get 
$\cln (T_1^*)$ is $T_2$-reducing subspace. Hence $T_2|_{\cln (T_1^*)}$ 
is an isometry. Now
\begin{align*}
	(I-T_2T_2^*)[\cln (T_1^*)]
	&=\cln (T_1^*) \ominus T_2 \cln( T_1^*) \\
	&=\clr(I - T_1 T_1^*) \ominus \clr [T_2(I - T_1 T_1^*)] \\ 
	&= \clr (I - T_1 T_1^*) \ominus \clr [T_2(I - T_1 T_1^*)T_2^*] \\ 
	&= \clr[(I - T_1 T_1^*) - T_2(I - T_1 T_1^*)T_2^*] \\
	&= \clr[(I - T_1 T_1^*) (I - T_2 T_2^*)] \\
	&= \clr(I - T_1 T_1^*) \cap \clr(I - T_2 T_2^*)   \\ 
	&= \cln (T_1^*) \cap \cln (T_2^*).
\end{align*}
Thus for each fixed $ m_2 \geq 1$, we can write
\begin{align*}
(I-T_2^{ m_2}T_2^{* m_2})[\cln (T_1^*)]
&= \cln ( T_1^*) \ominus T_2^{ m_2} \cln ( T_1^*) \\
&=[\cln  (T_1^*) \ominus T_2 \cln ( T_1^*))] \oplus \cdots \oplus 
T_2^{ m_2-1} [\cln ( T_1^*) \ominus T_2 \cln ( T_1^*)]\\
&= (\cln ( T_1^*) \cap \cln ( T_2^*)) \oplus \cdots \oplus 
T_2^{ m_2-1}( \cln ( T_1^*) \cap \cln ( T_2^*)) \\
&= \bigoplus_{k_2=0}^{ m_2-1}  T_2^{k_2}( \cln ( T_1^*) \cap \cln ( T_2^*)).
\end{align*}
Therefore,
\begin{align*}
\clh_{\neg u\neg u} &=  \bigvee_{  m_2 \geq 1 }  \bigg\{ 
\bigoplus_{ m_1=0}^{\infty}   \Big( \bigoplus_{k_2=0}^{ m_2-1}  
T_1^{ m_1}T_2^{k_2} \big( \cln ( T_1^*) \cap \cln ( T_2^*) \big) \Big) \bigg\}\\
&= \bigoplus_{ m_1=0}^{\infty}  \Big( \bigoplus_{ m_2=0}^{\infty}  
T_1^{ m_1}  T_2^{ m_2}  \big( \cln ( T_1^*) \cap \cln ( T_2^*) \big)  \Big) \\
&= \bigoplus_{ m_1, m_2=0}^{\infty}  T_1^{ m_1}  T_2^{ m_2} \big( \cln ( T_1^*) 
\cap \cln ( T_2^*) \big). 
\end{align*}

We know that the c.n.u. part of an isometry coincides with the
shift part of Wold-von Neumann decomposition. Hence we can obtain the
following decomposition for pairs of {\it{doubly twisted isometries}}:

\begin{thm}\label{wold-type-pairs-isomteries}
Let $\left( T_1, T_2 \right)$ be a pair of {\it{doubly twisted isometry}} 
on a Hilbert space $\mathcal{H}$. Then there is a unique decomposition 
\[
\mathcal{H}=\mathcal{H}_{uu} \oplus \mathcal{H}_{us} \oplus \mathcal{H}_{su} \oplus \mathcal{H}_{ss},
\]
where $\mathcal{H}_{uu},  \mathcal{H}_{us}, \mathcal{H}_{su},$ and 
$\mathcal{H}_{ss}$ are the subspaces reducing $T_1,  T_2$ such that
\begin{itemize}
\item	$T_1|_{\mathcal{H}_{uu}}, T_2|_{\mathcal{H}_{uu}}$ are unitary operators,
		
\item   $T_1|_{\mathcal{H}_{us}}$ is unitary, $T_2|_{\mathcal{H}_{us}}$ is unilateral shift,
		
\item  $T_1|_{\mathcal{H}_{su}}$ is unilateral shift, $T_2|_{\mathcal{H}_{su}}$ is unitary, 
		
\item $T_1|_{\mathcal{H}_{ss}}$, $T_2|_{\mathcal{H}_{ss}}$ are unilateral shifts. 
\end{itemize} 	
Also
\begin{align*}
		&\clh_{uu} = \bigcap_{ m_1, m_2 \in \mathbb{Z}_+} T_1^{ m_1} T_2^{ m_2} \clh,  
		&\clh_{us}= \bigoplus_{ m_2=0}^{\infty} T_2^{ m_2}  
		\Big (\bigcap_{ m_1=0}^{\infty} T_1^{ m_1}   \cln ( T_2^*)   \Big),\\
		& \clh_{su}=\bigoplus_{ m_1=0}^{\infty}  T_1^{ m_1} 
		\Big( \bigcap_{ m_2=0}^{\infty} T_2^{ m_2}  \cln ( T_1^*)  \Big), 
		&\clh_{ss}= \bigoplus_{ m_1, m_2=0}^{\infty}  T_1^{ m_1}  T_2^{ m_2}  
		\Big( \cln ( T_1^*) \cap \cln ( T_2^*) \Big).
\end{align*}
\end{thm}

\begin{rem}
The above result recovers the Wold-type decomposition and 
its orthogonal spaces for pairs of doubly commuting isometries 
(In particular, if we take twist $\clu_2=\{I\}$, an 
identity operator) on Hilbert spaces which were firstly studied by
S\l oci\'{n}ski \cite{SLOCINSKI-WOLD} and later for pairs of commuting 
isometries by Popovici \cite{POPOVICI-WOLD TYPE}. It is also a noteworthy 
to mention that orthogonal decomposition spaces of Wold-type decomposition 
for pairs of doubly twisted isometries and for pairs of doubly commuting 
isometries on Hilbert spaces are the same.
\end{rem}

%%%%%%%%%%%%%%%%%%%%%%%%%%%%%%%%%%%%%%%%%%%%%%%%
%%%%%%%%%%%%%%%%%%%%%%%%%%%%%%%%%%%%%%%%%%%%%%%
%%%%%%%%%%%%%%%%%%%%%%%%%%%%%%%%%%%%%%%%%%%%%%%%%%%
\section{Decomposition for $\clu_n$-twisted contractions}

In this section, we will find the explicit Wold-type decomposition for
$\clu_n$-twisted contractions on Hilbert spaces. As a by-product, we 
derive a simple proof for Wold-type decomposition for $\clu_n$-twisted 
isometries. 
\vspace{0.1cm}

Before proceeding further, we shall introduce certain notations 
for the remainder of the paper. Given an integer $p$ for $1 \leq p \leq n$, 
we denote the set $\{1,\ldots, p\}$ by $I_p$ and $A_p \subseteq 
I_p$; that means each subset (including empty set) of $I_p$  
is denoted by $A_p$. In addition, if $A_p \subseteq I_p$ and 
$q \notin I_p$, then $\tilde{A}_p$ is denoted by same $A_p$ but we 
will treat $\tilde{A}_p$ as a subset of $I_{p} \cup \{q\}$. Using 
the aforementioned notations, we generalize decomposition for
$\mathcal{U}_n$-twisted contraction as follows:

\begin{rem}\label{Generalized-canonical-decomposition}
Let $T_q$ be a contraction on a Hilbert space $\clh_{A_p}$, where 
$A_p \subseteq I_p$ and $q \notin I_p$. Then $\clh_{A_p}= \clh_{\tilde{A}_p} 
\oplus \clh_{{\tilde{A}_p} \cup \{q\}}$ is the orthogonal decomposition 
for $T_q$ such that $T_q|_{\clh_{\tilde{A}_p}}$ is unitary and 
$T_q|_{\clh_{{\tilde{A}_p} \cup \{q\}}}$ is c.n.u. In particular,
if $T$ is a single contraction on $\clh$, then from the 
above notation the canonical decomposition for $T$ is $\clh=\clh_{\emptyset} 
\oplus \clh_{\{1\}}$, where $\clh_{\emptyset} =\clh_{u}$ and $\clh_{\{1\}} 
=\clh_{\neg u}$ (see Theorem \ref{Wold-type-contraction}).
\end{rem}

We are now in a position to state our main result in this section.

\begin{thm}	\label{thm-Multi-contraction}
Let $n \geq 2$, and let $(T_1, \ldots, T_n )$ be a $\mathcal{U}_n$-twisted 
contraction on a Hilbert space $\mathcal{H}$. Then there exists $2^n$-joint 
$(T_1, \ldots, T_n)$-reducing subspaces $\{ \mathcal{H}_{A_n}: A_n 
\subseteq I_n$ (counting the trivial subspace $\{0\}$) such that 
\begin{align}
\mathcal{H}=\bigoplus_{A_n \subseteq I_n} \mathcal{H}_{A_n}
\end{align}
where
\begin{align}\label{subspace-equn 2}
\mathcal{H}_{A_n}=
\begin{cases}
& \bigcap_{ m_n \in \mathbb{Z}_+} \big[\mathcal{N} ((I -T_n^{* m_n}
T_n^{ m_n})|_{\mathcal{H}_{A_{n-1}}}) 
\cap \mathcal{N} ((I -T_n^{ m_n} T_n^{* m_n})|_{\mathcal{H}_{A_{n-1}}})\big] 
\quad 	\text{for } \ \ n \notin A_n, \\
& \bigvee_{  m_n \in \mathbb{Z}_{+}} \big\{ (I-T_n^{* m_n}T_n^{ m_n}) 
\mathcal{H}_{A_{n-1}} \cup (I-T_n^{ m_n}T_n^{* m_n}){\mathcal{H}_{A_{n-1}}} \big\}
\quad \qquad \quad \	\text{for} \ \ n \in A_n. 
\end{cases}
\end{align}
For each $A_n$ and $\mathcal{H}_{A_n} \neq \{0\}$, $T_i|_{\mathcal{H}_{A_n} }$ 
is unitary if $i \notin A_n$ and $T_i|_{\mathcal{H}_{A_n} }$ is completely 
non-unitary if $i \in A_n$ for $i= 1, \ldots, n$. Moreover, the above 
decomposition is unique.
\end{thm}

\begin{proof}
We will prove this by mathematical induction. 
Suppose that $(T_1,T_2)$ is a pair of $\clu_2$-twisted contraction 
on $\clh$. Then by Theorem \ref{Wold-type-pair-contractions}, there 
exists four $(T_1, T_2)$-reducing subspaces $\clh_{uu}, \clh_{u \neg u}, 
\clh_{\neg u u}$, and $\clh_{ \neg u \neg u}$ of $\clh$ such that
\begin{align*}
\clh  &=\mathcal{H}_{uu} \oplus \mathcal{H}_{\neg u u} \oplus 
\mathcal{H}_{u \neg u } \oplus \mathcal{H}_{\neg u \neg u} \\
& =\mathcal{H}_{\emptyset} \oplus \mathcal{H}_{\{1\}} \oplus 
\mathcal{H}_{\{2\}} \oplus \mathcal{H}_{\{1,2\}} \\ 
&=\bigoplus_{A_2 \subseteq I_2}\clh_{A_2},
\end{align*}
where $\clh_{A_2}$ has the explicit form (see Theorem 
\ref{Wold-type-pair-contractions}). Moreover, $T_i|_{\clh_{A_2}}$ 
is unitary if $i \notin A_2$ and c.n.u. if $i \in A_2$ for $i=1, 2$. 
So the statement is true for $n=2$.

Assume that the statement is true for any $k$-tuple $(T_1, \ldots, T_k)$, 
$k <n$ of $\clu_k$-twisted contraction on $\clh$. Then 	
\begin{align*}
\clh &=\bigoplus_{\substack{{A_{k} \subseteq  I_k}}} \clh_{A_k},
\end{align*}
where 
\begin{align} \label{Main-Unitary-and-cnu-space}
\mathcal{H}_{A_k}=
\begin{cases}
& \bigcap_{ m_k \in \mathbb{Z}_+} \big[\mathcal{N} ((I -T_k^{* m_k}
T_k^{ m_k})|_{\mathcal{H}_{A_{k-1}}}) 
\cap \mathcal{N} ((I -T_k^{ m_k} T_k^{* m_k})|_{\mathcal{H}_{A_{k-1}}})\big] 
\quad 	\text{for } \ \ k \notin A_k \\
& \bigvee_{  m_k \in \mathbb{Z}_{+}} \big\{ (I-T_k^{* m_k}T_k^{ m_k}) 
\mathcal{H}_{A_{k-1}} \cup (I-T_k^{ m_k}T_k^{* m_k}){\mathcal{H}_{A_{k-1}}} \big\}
\quad \qquad \quad \	\text{for} \ \ k \in A_k. 
\end{cases}
\end{align}
Also $T_i|_{\clh_{ A_{k}}}$ is unitary if $i \notin A_{k}$ and is c.n.u. 
if $i \in A_{k}$ for $i= 1, \ldots, k$. 
It is to be noted that the spaces $\mathcal{H}_{A_{k-1}}$ reduce $T_k$
from Lemma \ref{commuting operators-cor1}, and the decomposition for $\clu_k$-twisted 
contraction yields $2^{k}$ number of orthogonal $T_i$-reducing subspaces 
$\clh_{A_k}$ for $1 \leq i \leq k$. We shall now prove this statement for the 
decomposition of $(k+1)$-tuple $(T_1, \ldots, T_{k+1})$ of $\clu_{k+1}$-twisted 
contraction on $\clh$. Indeed, we show that 
\begin{align*}
\clh &=\bigoplus_{\substack{{A_{k+1} \subseteq  I_{k+1}}}}\clh_{A_{k+1}}.
\end{align*}
As the tuple $(T_1, \ldots, T_n)$ is $\clu_n$-twisted contraction, using
Lemma \ref{commuting operators-cor1} and the equation \eqref{Main-Unitary-and-cnu-space}, 
we have $T_{j} \clh_{A_k} \subseteq \clh_{A_{k}}$ and $T_{j}^* \clh_{A_k}
\subseteq \clh_{A_{k}}$ for all $k <j \leq n$, that is, $\clh_{A_{k}}$ reduce $T_j$ 
for $k <j \leq n$. Therefore, the canonical decomposition for the contraction 
$T_{k+1}|_{\clh_{A_k}}$ yields ${\clh_{A_k}}= {\clh_{\tilde{A}_k}} \oplus 
{\clh_{\tilde{A}_k \cup \{k+1\}}}$ (see Remark \ref{Generalized-canonical-decomposition}), 
where $\tilde{A}_k=A_k$ but as a subset of $I_k \cup \{k+1\}=I_{k+1}$. Moreover, 
$T_{k+1}|_{\clh_{\tilde{A}_k} }$ is unitary and $T_{k+1}|_{\clh_{\tilde{A}_k \cup \{k+1\}}}$ 
is c.n.u., that is, $T_{k+1}|_{\clh_{{A}_{k+1}} }$ is unitary if $k+1 \notin A_{k+1}$
and is c.n.u. if $k+1 \in A_{k+1}$. Consequently,
\begin{align*}
\clh
		&=\bigoplus_{\substack{A_k \subseteq I_k }} {\clh_{A_k}} \\
		&= 	\bigoplus_{\substack{ \tilde{A}_k=A_k \subseteq I_{k+1}}} 
		[{\clh_{\tilde{A}_k}} \oplus {\clh_{\tilde{A}_k \cup \{k+1\}}}] \\
		&=\bigoplus_{\substack{{A_{k+1} \subseteq I_{k+1}}}} \clh_{A_{k+1}},
\end{align*}
where the subspaces $\clh_{A_{k+1}}$ reduce each $T_i$ for $1 \leq i \leq k+1$. 
Also for each ${A_{k+1}}$ and $\clh_{{A_{k+1}}} \neq \{0\}$, $T_i|_{\clh_{ A_{k+1}}}$ is unitary 
if $i \notin A_{k+1}$ and is c.n.u.	if $i \in A_{k+1}$ for $i= 1, \ldots, k+1$. 
Since $T_{k+1}|_{\clh_{A_k}}$ is a contraction, from Theorem 
\ref{Wold-type-contraction} we obtain 
\[
	\clh_{A_{k+1}}
	= \bigcap_{ m_{k+1} \in \mathbb{Z}_+} 
	\Big[\cln \big((I -T_{k+1}^{* m_{k+1}} T_{k+1}^{ m_{k+1}})|_{\clh_{A_{k}}}\big) 
	\cap \cln \big((I-T_{k+1}^{ m_{k+1}} T_{k+1}^{* m_{k+1}})|_{\clh_{A_{k}}}\big)\Big] 
	\quad \text{if} \quad k+1 \notin A_{k+1}
\]
and
\[
\clh_{A_{k+1}}=
\bigvee_{  m_{k+1} \in \mathbb{Z}_+} \Big\{ (I-T_{k+1}^{* m_{k+1}}T_{k+1}^{ m_{k+1}}) 
\clh_{A_{k}} \cup (I-T_{k+1}^{ m_{k+1}}T_{k+1}^{* m_{k+1}}){\clh_{A_{k}}} \Big\}
\quad \text{if} \quad k+1 \in A_{k+1}. 
\]							
The uniqueness part of this decomposition comes from the uniqueness 
of the canonical decomposition of a contraction. 

This finishes the proof.	
\end{proof}

We shall now derive decomposition for $\clu_n$-twisted isometries.
More specifically, if an n-tuple of isometries $(T_1, \dots, T_n)$ 
on $\clh$ is a $\clu_n$-twisted isometry, then we obtain an explicit 
description of the orthogonal decomposition of $\clh$.
Before going to the proof let us adopt the following notations: 
Let $\Lambda =\{i_1<i_2<, \ldots, <i_{l-1}< i_l\} \subseteq I_n $ for
$1 \leq l \leq n$, $I_n \backslash \Lambda= \{i_{l+1}<i_{l+2}<, \ldots, 
<i_{n-1}< i_n\}$. The cardinality 
of the set $\Lambda$ is denoted by $|\Lambda|$. We denote by $T_{\Lambda}$ the 
$|\Lambda|$- tuple of isometries $(T_{i_1},\dots, T_{i_l})$ and 
$\mathbb{Z}_{+}^{\Lambda} :=\{ {\bf {m}}=(m_{i_1},\dots,m_{i_l}): m_{i_j}  \in 
\mathbb{Z}_{+}, 1 \leq j \leq l \}$. Also $T_{i_1}^{m_{i_1}} \cdots T_{i_l}^{m_{i_l}}$ 
is denoted by $T_{\Lambda}^ {\bf m}$ for $ {\bf m} \in \mathbb{Z}_{+}^{\Lambda}$.

Consider $\clw_{i_j}:= \clr(I-T_{i_j}T_{i_j}^*)=\cln (T_{i_j}^*)$ 
for each $1 \leq j \leq l$ and 
\[
\clw_{\Lambda}:= \clr\Big(\prod_{{i_j} \in \Lambda} \big(I-T_{i_j}T_{i_j}^*\big)\Big),
\]
where $\Lambda$ is a non-empty subset of $I_n$. We also denote $\clw_{\emptyset} = \clh$. 
As the tuple $(T_1, \ldots, T_n)$
is $\clu_n$-twisted isometry, Lemma \ref{commuting operators-cor1} implies that
$\big\{ (I- T_{i_j} T_{i_j}^*)\big\}_{{j}=1}^{l}$ is a family of commuting 
orthogonal projections. Therefore,
\[
\clw_{\Lambda}= \clr \Big(\prod_{{i_j} \in \Lambda} (I-T_{i_j}T_{i_j}^*)\Big)
= \bigcap_{{i_j} \in \Lambda} \clr \big(I-T_{i_j}T_{i_j}^*\big)
= \bigcap_{{i_j} \in \Lambda} \clw_{i_j}
\]
for each subset $\Lambda$ of $I_n$.

The following result is similar to Theorem 3.6 in \cite{{RSS-TWISTED ISOMETRIES}},
and Theorem 3.1 in \cite{SARKAR-WOLD} (For $n$-tuple of doubly commuting isometries).
However, our approach is different and derived from our above result Theorem 
\ref{thm-Multi-contraction} and properties of isometries. We are now in
a position to state the result as follows:

\begin{thm}\label{thm-Multi-isometries}
Let $n \geq 2$, and let $T=(T_1,\dots, T_n)$ be a $\clu_n$-twisted isometry on $\clh$. 
Then there exists $2^n$ joint $T$-reducing subspaces $\{\clh_{\Lambda}: 
\Lambda \subseteq I_n  \}$ (counting the trivial subspace $\{0\}$) such that 
\begin{align*}
\clh= \bigoplus_{\substack{ \Lambda \subseteq I_n }} \clh_{\Lambda}
\end{align*}
and for each $ \Lambda \subseteq I_n $, we have 
\begin{align}\label{equn 4}
\clh_{\Lambda}= \bigoplus_{{\bf k} \in \mathbb{Z}_{+}^{ \Lambda}} T_{{ \Lambda}}^{\bf k} 
\bigg(   \bigcap_{{\bf m} \in \mathbb{Z}_{+}^{{I_n \backslash \Lambda}}}T^{\bf m}
_{ I_n \backslash \Lambda} \ \clw_{ \Lambda}  \bigg).
\end{align}	
And for  $\clh_{\Lambda} \neq \{0\}$, $T_i|_{\clh_{\Lambda}}$ is unitary 
if $i \in I_n \backslash \Lambda$ and $T_i|_{\clh_{\Lambda}}$ is shift 
if $i \in  \Lambda $ for all $i=1, \dots, n$. Furthermore, the above 
decomposition is unique.
\end{thm}

\begin{proof}
Let $n \geq 2$ be a fixed integer. Suppose that $T=(T_1,\dots, T_n)$ is 
a $\clu_n$-twisted isometries on $\clh$. Thus, by Theorem 
\ref{thm-Multi-contraction} there exists $2^n$ joint $T$-reducing subspaces 
$\{\clh_{A_n}: A_{n} \subseteq I_n \} $ (including the trivial subspace $\{0\}$) 
such that 
\begin{align*}
\clh= \bigoplus_{ A_n \subseteq I_n }  \clh_{A_n}.
\end{align*}
Moreover, for every non-zero decomposition spaces $\clh_{A_n} $, 
$T_i|_{\clh_{A_n}}$ is unitary if $i \notin A_n$ and is shift if $i \in A_n$ 
for each $i=1, \ldots, n$. Now if $n \notin A_n $, then by equation 
\eqref{subspace-equn 2}, the subspace $\clh_{A_n}$ becomes 
\begin{align*}
\clh_{A_n}
&= \bigcap_{ m_{n} \in \mathbb{Z_{+}}} \Big[ \cln \big((I -T_{n}^{ m_{n}} 
T_{n}^{*m_n})|_{\clh_{A_{n-1}}}\big)\Big]  \\
&=  \bigcap_{ m_n \in \mathbb{Z}_{+}} T_n^{ m_n} \clh_{A_{n-1}}.
\end{align*} 
Now consider $A_{i_j} \subseteq I_{i_j} = \{i_1<i_2<, \ldots, <i_{j-1}< i_j\}$
for $j=1, \ldots, l$ and $1 \leq l \leq n$.
For a fixed $l$, let $\Lambda =\{i_1<i_2<, \ldots, <i_{l-1}< i_l\}$. 
Suppose ${i_j} \notin A_{i_j}$ for $j=n, n-1, \ldots, l+1$. Since the tuple
$T=(T_1,\dots, T_n)$ is $\clu_n$-twisted and repeating the above step, 
we have 
\begin{align}\label{multi-iso-equ}
\clh_{\Lambda} 
&=  \bigcap_{{ m_{i_{l+1}}, \ldots, m_{i_n}} \in \mathbb{Z}_{+}} 
		T_{i_n}^{ m_{i_n}} \cdots
		T_{i_{l+1}}^{ m_{i_{l+1}}}\clh_{A_{i_{l}} }\\ \nonumber
&=\bigcap_{ {\bf m} \in \mathbb{Z}_{+}^{I_n \backslash \Lambda}}
		T_{I_n \backslash \Lambda}^{\bf m} \clh_{A_{i_l }}.  
\end{align}
Now for the remaining set $I_n \backslash \Lambda$ if  $i_j \in A_{i_j}$ 
for $j=1 , \dots, l$, then from equation \eqref{subspace-equn 2}, 
the subspace $ \clh_{{A_{i_{l}}} }$ can be expressed as	
\begin{align*}
\clh_{{A_{i_{l}}} } 
&= \bigvee_{  m_{i_l} \in \mathbb{Z}_{+}} \Big\{ \big(I-T_{i_l}
^{ m_{i_l}}T_{i_l}^{* m_{i_l}}\big){\clh_{A_{i_{l-1}} }} \Big\} \\
&= \bigvee_{ m_{i_{1}}, \ldots, m_{i_l} \in \mathbb{N}} 
\Big\{ \big(I-T_{i_l}^{ m_{i_l}}T_{i_l}^{* m_{i_l}}\big) \dots 
\big(I-T_{i_{2}}^{ m_{i_{2}}}T_{i_{2}}^{* m_{i_{2}}}\big)
\big(I-T_{i_{1}}^{ m_{i_{1}}}T_{i_{1}}^{* m_{i_{1}}}\big) {\clh} \Big\}. 
\end{align*}
Again applying Lemma \ref{commuting operators-cor1}, 
for $j=1, \ldots, l$ and $m_{i_j} \geq 1$, we can write
\begin{align*}
		&\big(I-T_{i_l}^{ m_{i_l}}T_{i_l}^{* m_{i_l}}\big) \cdots
		\big(I-T_{i_{2}}^{ m_{i_{2}}}T_{i_{2}}^{* m_{i_{2}}}\big)
		\big(I-T_{i_{1}}^{ m_{i_{1}}}T_{i_{1}}^{* m_{i_{1}}}\big) {\clh} \\
		&=\big(I-T_{i_l}^{ m_{i_l}}T_{i_l}^{* m_{i_l}}\big) \cdots
		\big(I-T_{i_{2}}^{ m_{i_{2}}}T_{i_{2}}^{* m_{i_{2}}}\big)
		\bigg[\bigoplus_{k_{i_1}=0}^{m_{i_{1}}-1} {T_{i_{1}}^{k_{i_1}}}
		\cln \big(T_{i_{1}}^*\big)\bigg] \\
		&=\bigoplus_{k_{i_1}=0}^{m_{i_{1}}-1}  {T_{i_{1}}^{k_{i_1}}} 
		\big(I-T_{i_l}^{ m_{i_l}}T_{i_l}^{* m_{i_l}}\big) \cdots
		\big(I-T_{i_{2}}^{ m_{i_{2}}}T_{i_{2}}^{* m_{i_{2}}}\big)
		\Big[\cln (T_{i_{1}}^*)\Big]\\
		&=  \bigoplus_{k_{i_1},k_{i_2}=0}^{m_{i_{1}}-1,m_{i_{2}}-1} 
		{T_{i_{1}}^{k_{i_1}}} {T_{i_{2}}^{k_{i_2}}} \big(I-T_{i_l}
		^{ m_{i_l}}T_{i_l}^{* m_{i_l}}\big) \cdots
		\big(I-T_{i_{3}}^{ m_{i_{3}}}T_{i_{3}}^{* m_{i_{3}}}\big) 
		\Big[\cln (T_{i_{1}}^*) \cap \cln ( T_{i_{2}}^*)  \Big]\\
		&=\bigoplus_{k_{i_1}, \ldots, k_{i_l}=0}^{m_{i_{1}}-1, \ldots,m_{i_{l}}-1}  
		{T_{i_{1}}^{k_{i_1}}} \cdots {T_{i_{l}}^{k_{i_l}}}
		\Big[\cln ( T_{i_{1}}^*) \cap \cdots \cap	\cln ( T_{i_{l}}^*)  \Big].
		%	   &=\bigoplus_{k \in \{0, \dots, m-1\}^{\Lambda'}; k=0}^{m-1} T_{\Lambda'}^k
\end{align*}
Hence, 
\begin{align*}
\clh_{\Lambda}	
		&=\bigcap_{ {\bf m} \in \mathbb{Z}_{+}^{I_n \backslash \Lambda}} 
		T_{I_n \backslash \Lambda}^{\bf m} \Big[  \bigvee_{ m_{i_{1}}, \dots, m_{i_l} 
		\in \mathbb{N}} \Big\{ \bigoplus_{k_{i_1}, \dots, k_{i_l}=0}^{m_{i_{1}}-1, 
		\dots,m_{i_{l}}-1} {T_{i_{1}}^{k_{i_1}}} \cdots {T_{i_{l}}^{k_{i_l}}}
		\big(\cln ( T_{i_{1}}^*) \cap \dots \cap	 \cln ( T_{i_{l}}^*)\big)  \Big\} \Big] \\
		&=  \bigvee_{ m_{i_{1}}, \dots, m_{i_l} \in \mathbb{N}}  \bigg\{ \bigoplus_{k_{i_1}, 
		\dots, k_{i_l}=0}^{m_{i_{1}}-1, \dots,m_{i_{l}}-1} {T_{i_{1}}^{k_{i_1}}} \cdots {T_{i_{l}}^{k_{i_l}}} 
		\Big[\bigcap_{ {\bf m} \in \mathbb{Z}_{+}^{I_n \backslash \Lambda}}T_{I_n \backslash \Lambda}
		^{ \bf m} \big(\cln ( T_{i_{1}}^*) \cap \dots \cap	 \cln ( T_{i_{l}}^*)\big)\Big] \bigg\} \\
		&=  \bigvee_{ m_{i_{1}}, \dots, m_{i_l} \in \mathbb{N}}  \bigg\{ \bigoplus_{k_{i_1}, 
		\dots, k_{i_l}=0}^{m_{i_{1}}-1, \dots,m_{i_{l}}-1} {T_{i_{1}}^{k_{i_1}}} \cdots {T_{i_{l}}^{k_{i_l}}} 
		\Big[\bigcap_{ {\bf m} \in \mathbb{Z}_{+}^{I_n \backslash \Lambda}}	
		T_{I_n \backslash \Lambda}^{ \bf m} \big(\clw_{ \Lambda}\big)\Big] \bigg\}  \\
		&=  \bigoplus_ { {\bf k} \in \mathbb{Z}_{+}^{\Lambda}} T_{ \Lambda}^{\bf k} 
		\Big(\bigcap_{ {\bf m} \in \mathbb{Z}_{+}^{I_n \backslash \Lambda}}	T_{I_n \backslash 
		\Lambda}^{ \bf m} \clw_{ \Lambda}\Big).	
\end{align*}
If $\Lambda = \emptyset \subseteq I_n$, then repeating the same step
as the equation (\ref{multi-iso-equ}) for $j=n, n-1,\ldots, 2, 1$, we obtain
\[
\clh_{\Lambda} =\bigcap_{ {\bf m} \in \mathbb{Z}_{+}^{I_n }}
		T_{I_n}^{\bf m} \clh. 
\]
Therefore, for any $\Lambda \subseteq I_n$, we have
\begin{align*}
\clh_{\Lambda}=  \bigoplus_ { {\bf k} \in \mathbb{Z}_{+}^{\Lambda}} T_{ \Lambda}^{\bf k} 
\Big(\bigcap_{ {\bf m} \in \mathbb{Z}_{+}^{I_n \backslash \Lambda}}	T_{I_n \backslash 
\Lambda}^{ \bf m} \clw_{ \Lambda}\Big).
\end{align*}	
Clearly, $T_i|_{\clh_{\Lambda}}$ is unitary for all $i \in I_n \backslash \Lambda $ 
and $T_i|_{\clh_{\Lambda}}$ is shift for all $i \in \Lambda$. The uniqueness part 
is coming from the uniqueness of the classical Wold decomposition of isometries. 
	
This completes the proof.
\end{proof}

%%%%%%%%%%%%%%%%%%%%%%%%%%%%%%%%%%%%%%%%%%%%%%%%%%%%%%%%%%%%%%%%%%%%%%%%%%%%%%
%%%%%%%%%%%%%%%%%%%%%%%%%%%%%%%%%%%%%%%%%%%%%%%%%%%%%%%%%%%%%%%%%%%%%%%%%%%%%%%%
%%%%%%%%%%%%%%%%%%%%%%%%%%%%%%%%%%%%%%%%%%%%%%%%%%%%%%%%%%%%%%%%%%%%%%%%%%%%%%
%%%%%%%%%%%%%%%%%%%%%%%%%%%%%%%%%%%%%%%%%%%%%%%%%%%%%%%%%%%%%%%%%%%%%%%%%%%%%%%%

\NI\textit{Acknowledgement:} 
The authors are thankful to the anonymous reviewer for 
his/her critical and constructive reviews and suggestions 
that have substantially improved the presentation of the paper.
The second author's research work is supported in part by the 
Mathematical Research Impact Centric Support (MATRICS)
(MTR/2021/000695) and the Core Research Grant (CRG/2022/006891), 
SERB (DST), Government of India.

\label{Ref}

\end{document}